\documentclass[12pt,reqno]{amsart}
\usepackage{amssymb,amsmath,amstext,amsthm,amsfonts}
\usepackage[dvips]{graphicx}
\usepackage[pdftex]{hyperref}
\usepackage[ansinew]{inputenc}
\usepackage{a4wide}
\usepackage[usenames]{color}

\newcommand{\NE}{\operatorname{NE}}
\newcommand{\PR}{\operatorname{PR}}
\newcommand{\Fix}{\operatorname{Fix}}
\newcommand{\CR}{\operatorname{CR}}

\newcommand{\lrd}{\operatorname{lrd}}

\newcommand{\inter}{\operatorname{int}}

\newcommand{\ord}{\operatorname{ord}}

\def \RR {{\mathbb R}}
\def \NN {{\mathbb N}}

\def \ve {\varepsilon}

\def \cn {\mathcal{N}}

\def \cc {\mathcal{C}}

\def \cb {\mathcal{B}}

\def \cs {\mathcal{S}}
\newcommand{\dem}{\begin{proof}}
\newcommand{\cqd}{\end{proof}}

\newtheorem{theorem}{Theorem}
\newtheorem{corollary}[theorem]{Corollary}

\newtheorem{Claim}{Claim}

\newtheorem{T}{Theorem}[section]

\newtheorem{Lemma}[T]{Lemma}

\theoremstyle{remark}
\newtheorem{Definition}{Definition}

\numberwithin{equation}{section}

\begin{document}

\title[Explosion of smoothness for conjugacies]{Explosion of smoothness for conjugacies between multimodal maps}
\author{José F. Alves}
\address{José F. Alves \\ Centro de Matemática da Universidade do Porto \\
Rua do Campo Alegre 687, 4169-007 Porto, Portugal}
\email{jfalves@fc.up.pt}
\urladdr{http://www.fc.up.pt/cmup/jfalves}

\author{Vilton Pinheiro}
\address{Vilton Pinheiro \\Departamento de Matem\'atica, Universidade Federal da Bahia\\
Av. Ademar de Barros s/n, 40170-110 Salvador,
 Brazil.}\email{viltonj@ufba.br}

\urladdr{http://www.pgmat.ufba.br}

\author{Alberto A. Pinto}
\address{Alberto A. Pinto \\ LIAAD-INESC Porto LA e Departamento de Matemática, Faculdade de Ci\^encias do Porto
\\Rua do Campo Alegre 687, 4169-007 Porto,
Portugal}\email{aapinto@fc.up.pt}

\urladdr{http://www.ccog.up.pt/index.php/alberto-adrego-pinto}

\date{\today}
\thanks{José F. Alves would like to thank Calouste Gulbenkian Foundation, the European Regional Development Fund through the programme COMPETE and FCT under the projects PEst-C/MAT/UI0144/2011 and PTDC/MAT/099493/2008 for their financial support.
Vilton Pinheiro would like to thank MCT/CNPq 14/2009, CNPq/Jovens Pesquisadores 2008 e  IMCTMAT for their financial support.
Alberto Pinto would like to thank LIAAD-INESC Porto LA, Calouste Gulbenkian Foundation, PRODYN-ESF, POCTI and POSI by FCT and Minist\'erio da Ci\^encia e da Tecnologia, and the FCT
Pluriannual Funding Program of LIAAD-INESC Porto LA for their financial support.
}

\setcounter{tocdepth}{1}

\begin{abstract}
Let $f$ and $g$ be smooth multimodal maps with no periodic attractors  and no  neutral points.
If a  topological conjugacy $h$ between $f$ and $g$ is
$C^{1}$  at a point  in the nearby expanding set of $f$,
then $h$ is a smooth  diffeomorphism
 in the   basin  of attraction of a renormalization interval  of $f$.
 In particular, if $f:I \to I$ and $g:J \to J$ are $C^r$ unimodal maps and $h$ is $C^{1}$ at a boundary of $I$ then $h$ is $C^r$ in $I$.
 \end{abstract}

\maketitle


\section{Introduction}

 There is a well-known theory in hyperbolic dynamics that studies properties of the dynamics and of the topological conjugacies
  that lead to   additional regularity for the conjugacies.
  D. Mostow \cite{Mostow} proved that if  $\mathbb{H}/\Gamma_X$ and  $\mathbb{H}/\Gamma_Y$ are two closed hyperbolic Riemann surfaces covered by finitely generated Fuchsian groups $\Gamma_X$ and $\Gamma_Y$ of finite analytic type, and
$\phi:\overline{\mathbb{H}} \to \overline{\mathbb{H}}$ induces the isomorphism
$i(\gamma)= \phi \circ  \gamma \circ \phi^{-1}$, then $\phi$ is a M\"obius transformation if, and only if, $\phi$  is absolutely continuous.
   M. Shub and D. Sullivan \cite{5SSSShhhub} proved  that for any two analytic orientation
preserving circle expanding endomorphisms $f$ and $g$ of the same degree,  the conjugacy is analytic if, and only if, the conjugacy is absolutely continuous.
 Furthermore, they proved    that if  $f$ and $g$ have the same set of eigenvalues, then the conjugacy is analytic.
R. de la Llave  \cite{12}  and J.M. Marco and R. Moriyon
\cite{14,15}
 proved that   if  Anosov diffeomorphisms have the same set of eigenvalues, then
the conjugacy is smooth. For maps with critical points,
M. Lyubich \cite{LLL1} proved that $C^2$ unimodal maps with Fibonnaci
 combinatorics and   the same eigenvalues are $C^1$ conjugate.
W. de Melo and  M.  Martens \cite{MMMMM1} proved that if topological conjugate unimodal maps, whose attractors are cycles of intervals, have the same set of eigenvalues, then the conjugacy is smooth.    N. Dobbs \cite{Dobbs} proved that if a  multimodal map $f$ has an   absolutely continuous invariant measure, with a positive Lyapunov exponent, and   $f$ is absolutely continuous conjugate to another multimodal map, then the conjugacy is  $C^r$ in the domain of some induced Markov map of $f$.

Here, we  study  the explosion of smoothness for topological conjugacies,
i.e.   the conditions under which the smoothness of the conjugacy in a single point extends to an open set.
  P. Tukia \cite{Tukia} extended the result above of D. Mostow  proving that if  $\mathbb{H}/\Gamma_X$ and  $\mathbb{H}/\Gamma_Y$ are two closed hyperbolic Riemann surfaces covered by finitely generated Fuchsian groups $\Gamma_X$ and $\Gamma_Y$ of finite analytic type, and
$\phi:\overline{\mathbb{H}} \to \overline{\mathbb{H}}$ induces the isomorphism
$i(\gamma)= \phi \circ  \gamma \circ \phi^{-1}$, then $\phi$ is a M\"obius transformation if, and only if,  $\phi$ is differentiable at one radial
limit point with non-zero derivative.
Sullivan  \cite{sullivan} proved that if a topological conjugacy between analytic orientation preserving circle expanding endomorphisms of the same degree is differentiable at a  point with non-zero derivative, then
the conjugacy is analytic.
Extensions  of these results  for Markov maps and   hyperbolic basic sets on surfaces
 were developed by E. Faria  \cite{Edson111},  Y. Jiang \cite{JiangRG,111jia}  and A. Pinto, D. Rand and F. Ferreira \cite{FP,RP}, among others.
 For maps with critical points,
Y. Jiang \cite{4Jiang,5Jiang,6Jiang,11Jiang}  proved that quasi-hyperbolic one-dimensional maps   are smooth conjugated in an open set with full Lebesgue measure
if the conjugacy is differentiable at a  point with  uniform bound.
  In this paper, we  define  the nearby expanding set $\NE(f)$ of a multimodal  map $f$ and  characterize $\NE(f)$ in terms of the
  basins of attraction of renormalization intervals.
 We prove that  if a topological conjugacy between multimodal maps is
$C^{1}$   at a point  in the nearby expanding set $\NE(f)$ of $f$, then the conjugacy is a smooth diffeomorphism in
the basin of attraction of a renormalization interval.

\section{Explosion of smoothness}

Let $I$ be a compact interval and $f:I\to I$ a $C^{1+}$ map. By $C^{1+}$ we mean that $f$ is a differentiable map whose derivative is Hölder.
We say that $c$ is a \emph{non-flat turning point} of $f$, if there
exist $\alpha>1$ and a $C^{r}$ diffeomorphism $\phi$ defined in a
small neighborhood $K$ of $0$ such that
\begin{equation}\label{naodeg}
   f(c+x)=f(c) +\phi(|x|^{\alpha}),\quad\text{for every $x\in K$.}
\end{equation}
We say that $\alpha$ is the \emph{order} of the turning point  $c$
and denote it by $\ord_f(c)$.  We say that $f$ is a
\emph{multimodal} map if the next three conditions hold: \emph{i)} $f(\partial I)\subset \partial I$; \emph{ii)} $f$ has a finite number of turning points points that are all non-flat; and
 \emph{iii)} $\# \Fix(f^{n})<\infty$ for all $n\in\NN$.
A \emph{unimodal} map $f:I \to I$ is a non-flat multimodal map
with a unique turning point $c \in I$.

 The \emph{non-critical backward orbit} ${\mathcal O}^-_{nc} (p)$ of $p$ is the set of all points $q$ such that there is $n=n(q) \ge 0$ with the property that $f^n(q)=p$ and
$(f^n)'(q)\ne 0$. The \emph{non-critical alpha limit set} $\alpha_{nc} (p)$ of $p$ is the set of all accumulation points of  ${\mathcal O}^-_{nc} (p)$.
Let  ${\mathcal O}^-_{nc} (\PR(f))$  be the union  $\cup_{p \in \PR(f)} {\mathcal O}^-_{nc} (p)$ of the non-critical backward orbits ${\mathcal O}^-_{nc} (p)$
 for every repellor periodic points of $p\in \PR(f)$.
 Let $\alpha_{nc}(\PR(f))$ be the union $\cup_{p \in \PR(f)} \alpha_{nc} (p)$ of the non-critical alpha limit sets $\alpha_{nc} (p)$
for all repellor periodic points of $p\in \PR(f)$.

 A set  $A \subset J$ is said to be {\em forward invariant} if $f(A)
\subset A$. The {\em basin} $\cb(A)$ of a forward invariant set $A$ is
the set of all points $x \in A$ such that its omega limit set
$\omega(x)$ is contained in $A$.
An invariant compact set $A
\subset J$ is called a {\em (minimal) attractor}, in Milnor's sense \cite{milnor111,milnor2222}, if the
Lebesgue measure of its basin is positive and there is no forward invariant compact set $A'$ strictly contained in $A$ such that $\cb(A')$ has non zero measure.
The attractors of a $C^r$ non-flat multimodal map are of one of the following three types: i) a periodic attractor; ii) a minimal set with zero Lebesgue measure; or
iii) a cycle of intervals such that the omega limit set of almost
every point in the cycle is the whole cycle (see \cite{SV}).
 According to S. van Strien and E. Vargas \cite{SV}, if
$f:I \to I$ is a $C^r$ non-flat multimodal map, then  there is a
finite set of attractors $A_1,\dots,A_l\subset I$ such that the
union of their basins has full Lebesgue measure in $I$.

An open  interval $J(c)$ containing a critical point $c$ is a \emph{renormalization interval} of a multimodal (resp. unimodal) map $f$, if there is $n=n(J(c)) \ge 1$ such that  $f^n|_{\overline{J(c)}}$ is also a multimodal (resp. unimodal) map.
Hence, the forward orbit of $J(c)$ is a positive invariant set.
A multimodal map $f$ is \emph{no renormalizable inside  a renormalization interval} $J(c)$, if there is no renormalization interval strictly contained in $J$. A multimodal map  $f$ is  \emph{infinitely renormalizable around a critical point} $c$  if there is an infinite sequence of renormalization intervals $J_1(c), J_2(c), \ldots$ such that $J_{n+1}(c)$ is strictly contained in $J_n(c)$ and $c = \cap_{n \ge 1} J_n(c)$.
 The \emph{basin  of attraction} $\cb(J(c))$ \emph{of} $J(c)$ is   the set of points whose forward orbit intersects $J(c)$.

\begin{Definition} [Expanding and nearby expanding points]
A point $p\in I$ is called {\em nearby expanding} if  there are
\begin{enumerate}
\item a sequence of points $p_{n}$ converging to $p$,
\item a sequence of open intervals $V_{n}$ containing $p_{n}$,
\item a sequence of positive integers $k_{n}$ tending to infinity, and
\item $\delta=\delta(p)>0$,
\end{enumerate}
with the following properties:
\begin{enumerate}
\item $f^{k_{n}}|_{V_{n}}$ is a diffeomorphism and
\item  $f^{k_{n}}(V_{n})=B_{\delta}(f^{k_{n}}(p_{n}))$.
\end{enumerate}
Furthermore, a point $p\in I$ is called {\em  expanding} if $p\in I$ is a nearby expanding point with $p_n=p$ for every $n \in \mathbb{N}$.
\end{Definition}

The \emph{nearby expanding set} ${\NE}(f)$ is the set of all nearby expanding points of $f$ and
the \emph{expanding set} ${E}(f)$ is the set of all  expanding points of $f$.

\begin{Lemma} [Fatness of  ${E}(f)$ and ${\NE}(f)$]
\label{Expmap}
Let $f$  be $C^r$ a multimodal map with $r \ge 3$ and no periodic attractors  nor neutral periodic points. Then:
\begin{enumerate}
\item
$E(f)  \supset {\mathcal O}^-_{nc} (\PR(f))$ and $\NE(f)  \supset \alpha_{nc}(\PR(f))$;
\item
if $f$ is  infinitely renormalizable around a critical point $c$,
then there is a renormalization interval  $J(c)$ such that  ${E}(f)$ and ${\NE}(f)$ are  dense in $\cb (J(c))$;
\item
if $f$ is no renormalizable inside a renormalizable interval $J$,  then ${E}(f)$ is  dense in $\cb (J)$ and ${\NE}(f)$ contains $\overline{\cb (J)}$.
\end{enumerate}
\end{Lemma}

If  $f:I \to I$ is a unimodal map, for every renormalization interval $J$,      $\partial  \cb(J)$ is uniformly expanding,
   $\partial I \subset \partial\cb(J)$ and  $\cb(J)$ is an open set with full Lebesgue measure.
Hence, by Lemma \ref{Expmap}, if  $f$ is a unimodal map  whose attractor is a cycle of intervals then ${E}(f)$ is dense in $I$
and ${\NE}(f)=I$.
Furthermore, if  $f$ is a unimodal map that is infinitely renormalizable then ${E}(f)$ and ${\NE}(f)$ are dense in $I$.

 \bigskip
\begin{proof} Let $f$ be  infinitely renormalizable around a critical point $c$.
 By Lemma \ref{Expmap1}, there is a renormalization interval  $J(c)$ such that ${\mathcal O}^-_{nc} (\PR(f))$ is a dense set in $J(c)$.
  Since $E(f)  \supset {\mathcal O}^-_{nc} (\PR(f))$,
  we obtain that $E(f)$ and $\NE(f)$ are dense in $J(c)$.

Let  $f$ be no renormalizable inside a renormalizable interval $J$.
By Lemma \ref{Expmap1}, $\alpha_{nc}(\PR(f))$  contains $J$.
Hence, ${E}(f)$ is  dense in $J$ and ${\NE}(f)$ contains $J$.
\end{proof}

          \begin{Definition} [Puncture set $P(J)$]
   Let $C_P(I)$ be the set of all critical points $c$ whose non-critical alpha limit sets $\alpha_{nc} (c)$ do not intersect the interior of $I$.
   The  \emph{puncture set $P(I)$ of} $I$ is $P(I) = \cup_{c \in C_P(I)} {\mathcal O}^-_{nc} (c)$.
   Let $J$ be a renormalization interval and $n$ the smallest integer such that $F=f^n|J$ is a renormalization of $f$.
   Let $C_P(J)$ be the set of all critical points $c$ whose non-critical alpha limit sets $\alpha_{nc} (c)$ with respect to $F|J$ do not intersect the interior of $J$. The  \emph{puncture set $P(J)$ of} $J$ is $P(J) = \cup_{c \in C_P(J)} {\mathcal O}^-_{nc} (c)$.
    \end{Definition}

  Hence, the puncture set $P$ is either empty or a discrete set. Furthermore, we observe that the puncture set is not located in the central part of the dynamics, i.e. (i) if $f$ is infinitely renormalizable there is a renormalization interval $J(c)$ such that $P \cap J(c) = \emptyset$ and
  (ii) if the Milnor's  attractor $A$ of $f$ is a cycle of intervals then  $P \cap A= \emptyset$, because $\alpha_{nc} (c)$ is dense in $A$ for every critical point $c$ in the interior of $A$.

For every  connected component $G \in D(J)$, let $m=m(G) $ be the  smallest integer such that $f^m(G) \subset J(c)$.
If $m=0$  the puncture set $G_P \subset G$ of $G$ is $G_P= P(J)$, and
if $m > 0$ the puncture set $G_P \subset G$ of $G$ be the union of all  points $x \in {G}$ such that (i) $(f^m)'(x)=0$
or (ii) $(f^m)'(x) \in P(J)$.
 We observe that  $G_P \cap G$  is either a discrete set or empty.
The \emph{punctured basin  of attraction} $\cb_{P}(J(c))$ \emph{of} $J(c)$ is the union $\cup_{G \in D(J)}  G \setminus G_P$.
A renormalization domain $J= \cup_{c \in \CR} J(c)$ of a multimodal  map $f$ is the union of \emph{renormalization intervals} $J(c)$ for a given  subset  $\CR \subset C_f$.
Set    $\cb_P(J)=   \cup_{c \in \CR}  \cb_{P}(J(c))$. We observe that $\overline{\cb_P(J)}= \overline{\cb(J)}$.

 \begin{Definition} [$C^1$ at a point]
We say that a map $h:I \to I'$ is $C^1$ \emph{at a point} $p\in I$, if
$$
\lim_{x,y\to p\atop x\neq y}\frac{h(x)-h(y)}{x-y}=h'(p) \ne 0 .
$$
\end{Definition}

We observe that $h$ is  $C^{1}$  at every point belonging to an interval $K \subset I$ if, and only if, $f$ is a $C^1$ local diffeomorphism in that interval $K$.

We say that a topological conjugacy   $h:I \to L$  between
 $f:I \to I$ and $g:I' \to I'$ \emph{preserves the order of the critical points},
 if    $\ord_f(c)=\ord_g(h(c))$ for every  critical point $c \in C_f$.

\begin{theorem} [Explosion of smoothness]
\label{theoremAAA}
Let $f$ and $g$ be $C^r$ multimodal maps with $r \ge 3$ and no periodic attractors nor neutral periodic points.
Let $h$ be a topological conjugacy  between  $f$ and $g$  preserving the order of the critical points.
If $h$ is $C^{1}$ at a point $p \in \NE(f)$, then either
\begin{enumerate}
\item
$h$  is a $C^{r}$ diffeomorphism in $I\setminus P(I)$; or
\item
there is a unique maximal renormalization domain  $J$
such that $h$ is a $C^{r}$ diffeomorphism in $J \setminus P(J)$.
Furthermore,
\begin{enumerate}
\item $h$ is a $C^{r}$ diffeomorphism
 in  the punctured basin  of attraction $\cb_{P}(J)$;
 \item  $h$ is not  $C^{r}$ at any open interval contained in $I \setminus  \overline{\cb (J)}$;
\item  $h$ is not $C^{1}$ at any point in $E(f) \cap \partial  \cb(J)$.
 \end{enumerate}
  \end{enumerate}
\end{theorem}

  We observe that Theorem \ref{theoremAAA} still holds if we replace the  hypotheses of $h$ being  $C^{1}$ at a point $p \in E(f)$
  by  $h$ being  $C^{r}$  in an open set.
N. Dobbs \cite{Dobbs} proved that if (i) a multimodal map $f$ has an   absolutely continuous invariant measure with a positive Lyapunov exponent and (ii) the conjugacy $h$ between $f$ and another multimodal map $g$ is absolutely continuous, then $h$ is  $C^r$ in an open set. Hence, Theorem \ref{theoremAAA} applies to this case.

  The proof of Theorem \ref{theoremAAA} is given at the end of Section \ref{RRIIII3}.

\begin{corollary} [Full measure explosion of smoothness for unimodal maps]
\label{corollaryeee}
Let $f$ and $g$ be $C^r$ unimodal maps with $r \ge 3$ and no periodic attractors  nor  neutral periodic points.
Let~$h$ be a topological conjugacy  between  $f$ and $g$  preserving the order of the critical points.
If $h$ is $C^{1}$ at a point $p \in \NE(f)$, then either
\begin{enumerate}
\item
$h$ is a $C^{r}$ diffeomorphism  in the full interval $I$; or
\item
there is a unique maximal renormalization interval  $J \subseteq I$
such that
\begin{enumerate}
\item $h$ is a $C^{r}$ diffeomorphism
 in the     basin  $\cb(J)$,    and
\item  $h$ is not $C^{1}$ at any point in $\partial \cb(J)$.
 \end{enumerate}
  \end{enumerate}
\end{corollary}

  We observe that if  $f:I \to I$ is a unimodal map,  then  (i) $\partial  \cb(J)$ is uniformly expanding,
 (ii)  $\partial I \subset \partial\cb(J)$, and  (iii) $\cb(J)$ is an open set with full Lebesgue measure in $I$.
 By  Corollary  \ref{corollaryeee},  the map  $h$ is $C^{1}$ at a point $p \in  \partial I$
if, and only if, $h$ is  a $C^{r}$ diffeomorphism in $I$.

\section{Zooming pairs}

We will prove that, in Theorem \ref{theoremA} and in its two corollaries,
the hypothesis $h$ is $C^{1}$ at a point $p$
can be weakened to $h$ being (uaa) uniformly asymptotically affine
at $p$.
We will define the zooming pairs that we will use to show  if  $h$ is uaa at a point
then  $h$ and $h^{-1}$ are $C^{r}$   in small open sets.

 Let $h:I \to I'$ be a homeomorphism. For every  (x,y,z) of points $x,y,z \in I$,
 such that $x <y<z$, we define
the logarithmic ratio distortion $\lrd_h(x,y,z)$ by
$$
\lrd_h(x,y,z) =
 \left| \log
\frac{\left|h(z)-h(y)\right|}
{\left|h(y)-h(x)\right|}
\frac{\left|y-x\right|}
{\left|z-y\right|}
\right| \ .
$$

\begin{Definition} [uaa]
\label{desp1}
Let $h:I \to  I'$ be a homeomorphism. The map $h$ is  \emph{
uniformly asymptotically affine (uaa)}   at a point $p$  if,
for every $C \ge 1$,
there is a continuous function $\epsilon_C:\RR^+_0  \to
\RR^+_0$,  with $\epsilon_C (0)=0$, such that
\begin{equation}
\label{fgggrrs31313131}
\lrd_h(x,y,z)
\le \epsilon_C (|x-p|) \ ,
\end{equation}
for all $x < y< z$ with $C^{-1} < |z-y|/|y-x|<C$.
\end{Definition}

\begin{Lemma} [$C^{1}$ implies uaa]
\label{esp1}
Let $h:I \to I'$ be a homeomorphism.
If $h$ is $C^{1}$ at a point $p \in I$, then $h$ is uaa  at $p$.
\end{Lemma}

\begin{proof}
If $h$ is $C^{1}$ at $p$, then
 there is a sequence $\theta_m$ converging to $0$,
when $m$ tends to $\infty$, such that
\begin{equation}
\label{gfsge5}
\left| \log
\frac{\left|h(y)-h(x)\right|}
 {\left|y-x\right|}
h'(p)
\right|
\le O \left(\frac{1}{m}\right) \ ,
\end{equation}
for all $x,y  \in B_{\theta_m}(p)$. Hence, for all $x,y,z  \in
B_{\theta_m}(p)$, we obtain
\begin{equation}
\label{gfsge7}
\left| \log
\frac{\left|h(z)-h(y)\right|}
{\left|h(y)-h(x)\right|}
\frac{\left|y-x\right|}
{\left|z-y\right|}
\right|
\le O \left(\frac{1}{m}\right) \ ,
\end{equation}
and so, $h$ is uaa at $p$.
\end{proof}

\begin{Definition} [$\alpha$-bounded distortion]
We say that a $C^r$ multimodal map $f$ has  $\alpha$-\emph{bounded distortion}
with respect to a sequence $V_1,V_2,\ldots$ of  intervals and  a sequence of integers
 $k_n$ tending to $\infty$, if  there is $C \ge 1$ such that
\begin{equation}
\label{gfsge1}
\lrd_{f^{k_{n}}}(x,y,z)
\le C |f^{k_{n}}(z)-f^{k_{n}}(x)|^\alpha  \ ,
\end{equation}
 for  all $x,y,z \in V_{n}$, with  $x < y < z$, and all $n \ge 1$.
 \end{Definition}

 \begin{Definition} [Zooming pair $(p,V)$]
 \label{desp5}
Let $f:I\to I$ and $g:I'\to I'$ be $C^{r}$ maps, with $r\ge2$, and $h:I\to I'$ a topological conjugacy between $f$ and $g$. An $\alpha$-\emph{zooming pair} $(p,V)$ consists of a point $p\in I$ and an open interval $V\subset I$ such that
 \begin{enumerate}
 \item
  there is a sequence $V_1,V_2,\ldots$ of  intervals in $I$ and
   \item a sequence of integers $k_n$ tending to $\infty$,
   \end{enumerate}
   with the following  properties:
 \begin{enumerate}
 \item
  $\sup_{x \in V_n} |x-p| \to 0$ when $n \to \infty$;
  \item
  $f^{k_{n}}|_{V_n}$ and $g^{k_{n}}|_{h(V_n)}$ are diffeomorphisms onto the intervals $V$ and $h(V)$ respectively;
  \item $f$ has $\alpha$-bounded distortion
  with respect to the sequences $V_1,V_2,\ldots$ and $k_1,k_2,\ldots$;
  \item $g$ has $\alpha$-bounded distortion
  with respect to the  sequences $h(V_1)$, $h(V_2),\ldots$ and $k_1,k_2,\ldots$.
  \end{enumerate}
  A \emph{central zooming pair} $(p,V)$ is a zooming pair $(p,V)$ with the property that $p \in V_n$ for some $n \in \mathbb{N}$.
 \end{Definition}

\begin{Lemma}[Explosion of smoothness from $p$ to $V$]
\label{esp3}
Let $f$ and $g$ be $C^r$ maps, with $r \ge 3$, topologically conjugated by a homeomorphism $h$. Assume that $(p,V)$ is an $\alpha$-zooming pair for some $0 < \alpha <1$.
If   $h$ is uaa at $p$,
then  $h|V$ is a $C^{1 + \alpha}$ diffeomorphism onto its image.
Furthermore, if $(p,V)$ is a central zooming pair then $h|V_0$ is a $C^{1+\alpha}$ diffeomorphism onto its image, for some open interval $V_0$ containing $p$.
\end{Lemma}

\begin{proof}
Given $a,b,c \in V$, with $a<b<c$,
let $a_n,b_n,c_n \in V_n$ be such that $f^{k_n}(a_n)=a$,  $f^{k_n}(b_n)=b$
and  $f^{k_n}(c_n)=c$. Since $f$ has
$\alpha$-uniformly bounded distortion,
\begin{equation}
\label{bgfbvgfr1}
\lrd_{f^{k_{n}}}(a_n,b_n,c_n)
 \le
O(|c-a|^\alpha) .
\end{equation}
Since $g$ has  has
uniformly bounded distortion,
  we get
 \begin{equation}
\label{bgfbvgfr3}
\lrd_{g^{k_{n}}}(h(a_n),h(b_n),h(c_n)) \le
O(|h(c)-h(a)|^\alpha) \ .
 \end{equation}
 By the definition of zooming, there is  a sequence $\sigma_n \to 0$   such that,
 for all $x \in V_n$,
 \begin{equation}
\label{aascfjefjio3}
|x -p| <  \sigma_n \ .
 \end{equation}
Since  $f$ is  (uaa) at $p$,
by \eqref{fgggrrs31313131}, we have
$$
\lrd_{h}(a_n,b_n,c_n)
\le  \epsilon_C( \sigma_n) \ .$$
Hence,  by \eqref{aascfjefjio3}, there is $n$ large enough such that
\begin{equation}
\label{bgfbvgfr5}
\lrd_{h}(a_n,b_n,c_n)
\le  |c-a| \ .
\end{equation}
Combining  \eqref{bgfbvgfr1}, \eqref{bgfbvgfr3} and \eqref{bgfbvgfr5}, we have
\begin{eqnarray}
\label{vvvvvfffff3}
\lrd_{h}(a,b,c)
&\le&
\lrd_{g^{k_{n}}}(h(a_n),h(b_n),h(c_n)) +
\lrd_{h}(a_n,b_n,c_n)
 +
\lrd_{f^{k_{n}}}(a_n,b_n,c_n) \nonumber \\& \le &
O(|c-a|^\alpha+ |h(c)-h(a)|^\alpha) \ .
\end{eqnarray}
Therefore, the homeomorphism $h$ is quasi-symmetric in $V$. Hence,
 there is $\gamma > 0$, such that $h|V$ is $\gamma$-H\"older continuous.
Thus, we obtain that \eqref{vvvvvfffff3} is bounded by
$C_1|c-a|^{\alpha\gamma}$, for some $C_1 >1$.
Hence, by
\cite{SP}, we get that $h|V$  and $h^{-1}|h(V)$ are $C^{1+\alpha\gamma}$ maps.
Therefore, $|h(c) - h(a)| \le O(|c-a|)$ and, so,
\eqref{vvvvvfffff3} is also bounded by
$C_2|c-a|^{\alpha}$, for some $C_2 >1$.
Hence, again by
\cite{SP}, we get that $h|V$  and $h^{-1}|h(V)$ are $C^{1+\alpha}$ maps.

Furthermore, if $(p,V)$ is a central zooming pair then there is an open interval $V_0$ containing $p$ and an integer $n$ such that
$f^n|V_0$ is a $C^r$ diffeomorphism and $f^n(V_0) \subset V$. Hence $h|V_0= (g^n|h(V_0))^{-1} \circ h \circ f^n$ is a $C^r$ diffeomorphism.
\end{proof}

\begin{Lemma}[Building up smoothness from $C^{1+\alpha}$ to $C^r$]
\label{esp4}
Let $f$ and $g$ be $C^r$ maps, with $r \ge 3$, topologically conjugated by a homeomorphism $h$.
If $h|V$ is a $C^{1+\alpha}$ diffeomorphism in some open set $V$, then
$h|W$  is a $C^r$ diffeomorphism for some open set $W \subset V$.
\end{Lemma}

\dem
By Lemma \ref{Expmap1},  there is a reppelor $p \in I$ and  integers  $m$ and $l$  such that $p \in \inter(f^m(V))$ and $f^l(p)=p$.
Since $p$ is a reppelor there is an open interval $W \subset \inter(f^n(V))$ with $p \in W$ such that
$|f^{lj}(x)| > \lambda > 1$, for all $x \in W$.
Let $W_0, W_1, \ldots $ be a sequence of open intervals contained in $W$ such that (i) $f^l(W_{n+1})=W_{n}$,  (ii) $W_{n+1}  \subset W_n$,
and (iii) $|W_n| \to 0$ for every $n \ge 0$.
Let $i_n:W_n \to (0,1)$ be the affine map with the property that $i_n(W_n)=(0,1)$ and
let $f_n=i_0 \circ f^{nl} \circ i_n^{-1}$.
By Lemma E13 in \cite{RP}, there is $b > 0$ such that $\| \ln df_n\|_{C^{r-1}} \le b$, for every $n \ge 1$.
Hence, by Lemma E15 in \cite{RP}, there is a small $\epsilon >0$ and a subsequence $f_{k_n}$
  converging to a $C^r$ diffeomorphism  $\underline{f}:(0,1) \to (0,1) $ in the $C^{r-\epsilon}$ norm.

Let $W_n'=h(W_n)$ and $j_n:W_n' \to (0,1)$ be the affine map with the property that $j_n(W_n')=(0,1)$, for every $n \ge 1$.
Let $g_n=j_0 \circ g^{nl} \circ j_n^{-1}$.
By Lemma E13 in \cite{RP}, there is $b > 0$ such that $\| \ln dg_n\|_{C^{r-1}} \le b$, for all $n \ge 1$.
Hence, by Lemma E15 in \cite{RP}, there is a small $\epsilon >0$ and a subsequence $m_n$ of the sequence $k_n$ such that
$g_{m_n}$
  converges to a $C^r$ diffeomorphism $\underline{g}$ in the $C^{r-\epsilon}$ norm.

Let $h_n=j_n \circ h \circ i_n^{-1}$. Since $h$ is a $C^{1+\alpha}$ diffeomorphism, there is a sequence $\lambda_n$ tending to $1$ such that
$$\frac{|h_n(z)-h_n(y)|}{|h_n(y)-h_n(x)|} \frac{|y-x|}{|z-y|} \le \lambda_n$$
for all $x,y,z \in (0,1)$.
Hence, $\underline{h}=\lim h_n$ is an affine map.

We note that $h|W_0 = j_0^{-1}  \circ g_n  \circ h_n \circ f_n^{-1} \circ i_0$, for every $n \ge 1$.
Hence, $$h|W_0 = \lim j_0^{-1}  \circ g_{m_n}  \circ h_{m_n} \circ f_{m_n}^{-1} \circ i_0 =
j_0^{-1}  \circ \underline{g}  \circ \underline{h} \circ \underline{f}^{-1} \circ i_0 .$$
Since, $\underline{g}$, $\underline{h}$ and $\underline{f}$ are $C^r$ diffeomorphisms, we obtain that $h|W_0$ is a $C^r$ diffeomorphism.
\cqd

\section{Nearby expanding set}
\label{EEEXXX}

We will prove that for every nearby expanding point $p \in \NE(f)$
there is an open set $V$ such that $(p,V)$ is a  $1$-zooming pair.

Given any $K\subset\RR$ and $r>0$, set $B_{r}(K)=\bigcup_{p\in K}B_{r}(p)$, where $B_{r}(p)=(p-r,p+r)$.

Recall that the {\em
Schwarzian derivative of} $f$ in the complement of the critical
points is defined by
 $$
 \cs f:= \frac{f'''}{f'}-\frac32\left(\frac{f''}{f'}\right)^2.
 $$

\begin{Lemma} [Nearby expanding point originates a zooming pair]
\label{zoom6}
Let $f$ and $g$  be  $C^3$  multimodal maps topologically conjugated by $h$,
with no periodic attractors and no neutral periodic points.
For every $x\in {\NE}(f)$ there is an interval $V$ such that $(x,V)$ is a $1$-zooming pair.
Furthermore, for every $x\in {E}(f)$ there is an interval $V$ such that $(x,V)$ is a central $1$-zooming pair.
\end{Lemma}

\begin{proof} By \cite{SV}, there is $\gamma > 0$  such that,
for every point $x \in I$, with
$$f^n(x) \in \bigcup_{c \in C(f)} B_{\gamma}(c)
~~~\mbox{ and }~~~
g^n(h(x)) \in \bigcup_{c \in C(f)} h(B_{\gamma}(c))  , $$
we have  $Sf^{n+1}(x)< 0$ and $Sg^{n+1}(h(x))< 0$.

By Lemma~\ref{NICE}, one  find $\gamma_{0}<\gamma_{1}<\gamma_{2}<\gamma_{3}<\gamma_{4}<\gamma$ and nice sets $J_{0},J_{1},J_{2}$ such that $$B_{\gamma_{0}}(\cc_{f})\subset J_{0}\subset B_{\gamma_{1}}(\cc_{f})\subset B_{\gamma_{2}}(\cc_{f})\subset J_{1}\subset B_{\gamma_{3}}(\cc_{f})\subset B_{\gamma_{4}}(\cc_{f})\subset J_{2}\subset B_{\gamma}(\cc_{f}).$$
Let $J_{i}=\bigcup_{c\in\cc_{f}}J_{i}(c)$, $c\in J_{i}(c)=(a_{i}(c),b_{i}(c))$ for every $c\in\cc_{f}$ and $i=0,1,2$.

Given $x\in {\NE}(f)$, for some small $\delta>0$,     take a sequence of points $x_{j}\to x$ and intervals $W_{j}^{0}\ni x_{j}$ such that $f^{m_j}|_{\overline{W_{j}^{0}}}$ is a diffeomorphism and $f^{m_j}(W_{j}^{0})=B_{2\delta}(f^{m_j}(x_{j}))$ for $m_{j}\to\infty$.
Let $W_{j}\subset W_{j}^{0}$ be the interval such that $f^{m_{j}}(W_{j})=B_{\delta}(x_{j})$ and  let $L_{j}^{0}, R_{j}^{0}$ be the connected components of $W_{j}^{0}\setminus W_{j}$.

For every $j\ge1$, define $n_{j}$ as follows:  If $f^{i}(x_{j})\notin J_{1}$, for every $0\le i<m_{j}$, take $n_{j}=-1$; otherwise, take $n_{j}<m_{j}$ as the biggest integer such that $f^{i}(x_{j})\in J_{1}$.

Our goal is to obtain a sequence $j_{i}\to+\infty$ and intervals $V_{j_{i}}\subset W_{j_{i}}^{0}$ containing
$x_{j_{i}}$ with the following properties:
 $\inf_{i}|f^{m_{j_{i}}}(V_{j_{i}})|>0$ and   the ratio distortion of $f^{m_{j_{i}}}|_{V_{j_{i}}}$   uniformly bounded. If $n_{j}=-1$ take $V_{j}=W_{j}$. In this case, $|f^{m_{j}}(V_{j})|=2\delta$ and the bounded of the ratio distortion follows from Theorem~\ref{zoom1}, because $J_{1}\supset B_{\gamma_{2}}(\cc_{f})$. Thus, we   assume from now on that $n_{j}\ne-1$.

If $\liminf_{j} m_{j}-n_{j}<\infty$, let $V_{j}$ be the maximal interval such that  $x_{j}\in V_{j}\subset W_{j}$ and $f^{n_{j}}(V_{j}) \subset J_{2}$. Taking a subsequence, we   assume that there is $K>0$ such that $m_{j}-n_{j}\le K$ for every $j$.

Since $Df^{m_{j}}\ne0$ in $W_{j}$ and $f^{n_{i}}(W_{i})\cap J_{1}\ne\emptyset$ and by maximality of $V_{j}$, if $W_{j}\ne V_{j}$ then $$f^{n_{j}}(V_{j})\supset (c_{j}-\gamma_{4},c_{j}-\gamma_{3})~~~ {\rm or} ~~~ f^{n_{j}}(V_{j})\supset(c_{j}+\gamma_{3},c_{j}+\gamma_{4})$$ for some $c_{j}\in\cc_{f}$. In particular, $|f^{n_{j}}(V_{j})|\ge\gamma_{4}-\gamma_{3}$. Thus, there is $\ve>0$ such that,  for every $j$, $|f^{m_{j}}(V_{j})|>\ve>0$, $|f^{m_{j}}(V_{j})|=|f^{m_{j}}(W_{j})|=2\delta$ or $f^{m_{j}}(V_{j})$ is a finite iteration of an interval with length greater than $\gamma_{4}-\gamma_{3}$. Furthermore, since
\begin{equation}\label{eq098iu}|f^{m_{j}}(L_{j})|/|f^{m_{j}}(W_{j})|=|f^{m_{j}}(R_{j})|/|f^{m_{j}}(W_{j})|=1/2\,\,\,   {\rm for} {\rm \, every} \,j,
\end{equation}
we get that
$$\frac{|f^{n_{j}+1}(L_{j})|}{|f^{n_{j}+1}(V_{j})|}\ge\frac{|f^{n_{j}+1}(L_{j})|}{|f^{n_{j}+1}(W_{j})|}$$
 and
 $$\frac{|f^{n_{j}+1}(R_{j})|}{|f^{n_{j}+1}(V_{j})|}\ge\frac{|f^{n_{j}+1}(R_{j})|}{|f^{n_{j}+1}(W_{j})|}$$
  are bounded away from zero. Since $Sf^{n_{j}+1}(z)<0$, for every
  $$z\in f^{-n_{j}}(B_{\gamma}(\cc_{f}))\supset f^{-n_{j}}(B_{\gamma_{5}}(\cc_{f}))\supset f^{-n_{j}}(J_{2})\supset V_{j} ,$$
   the ratio distortion of $f^{n_{j}+1}|_{V_{j}}$ is uniformly bounded ($V_{j}\subset W_{j}$).
Thus, the ratio distortion of $f^{m_{j}}|_{V_{j}}$ is also uniformly bounded and $|f^{m_{j}}(V_{j})|>\ve>0$  for every $j$.

Let us consider the case $\liminf_{j} m_{j}-n_{j}=\infty$. Taking a subsequence, if necessary, we   assume that $\lim_{j} m_{j}-n_{j}=\infty$.

\begin{Claim}
$f^{n_{j}}(W_{j}^{0})\subset J_{2}$ for every $j \in \mathbb{N}$.
\end{Claim}
\dem[Proof of the claim]
Let $V_{j}^{0}$ be the maximal interval such that
$$x_{j}\in V_{j}^{0}\subset W_{j}^{0}~~~ {\rm and}~~~ f^{n_{j}}(V_{j}^{0}) \subset J_{2}.$$
 We will show that $W^{0}_{j}=V^{0}_{j}$.

By the maximality of $V_{j}^{0}$, if $W_{j}^{0}\ne V_{j}^{0}$ then there is $p_{2,j}\in\partial J_{2}\cap\partial (f^{n_{j}}(V_{j}^{0}))$.
On the other hand, since $f^{n_{j}}(x_{j})\in J_{1}$, there is $p_{1,j}\in\partial J_{1}$ such that
$$f^{n_{j}}(V_{j}^{0})\supset (p_{1,j},p_{2,j})~~~{\rm or}~~~f^{n_{j}}(V_{j}^{0})\supset(p_{2,j},p_{1,j}) . $$
If $p_{1,j}<p_{2,j}$ take $T_{j}=(p_{1,j},p_{2,j})$;  otherwise, take  $T_j=(p_{2,j},p_{1,j})$. Since $J_{1}$ and $J_{2}$ are nice sets with $J_{1}\subset J_{2}$, it follows that $f^{k}(\partial T_{j})\cap J_{1}=\emptyset$  for every $k\ge0$. Hence, if $\ell_{j}\ge0$ is the smaller integer such that $f^{\ell_{j}}(T_{j})\cap J_{1}\ne\emptyset$, then $f^{\ell_{j}}(T_{j})\cap J_{1}(c_{j})\ne\emptyset$ for some $c_{j}\in\cc_{f}$. Furthermore,  $f^{\ell_{j}}(T_{j})\supset J_{1}(c_{j})$. However, since $Df^{m_{j}}\ne0$ on $W_{j}^{0}$, we get $\ell_{j}\ge m_{j}-n_{j}$.
Thus, it follows from Theorem~\ref{zoom1} that $$4\delta=|f^{m_{j}}(W_{j}^{0})|\ge|f^{m_{j}}(V_{j}^{0})|\ge|f^{m_{j}-n_{j}}(T_{j})|\ge C\lambda^{m_{j}-n_{j}}|T_{j}|\ge$$ $$\ge C\lambda^{m_{j}-n_{j}}(\gamma_{4}-\gamma_{3})\to\infty\mbox{ (for a subsequence)}.$$ Hence, we get a  contradiction.
\cqd

By Theorem~\ref{zoom1}, if $f^{i}(W_{j}^{0})\cap J_{0}=\emptyset$, for every  $n_{j}<i<m_{j}$, then $f^{m_{j}-(n_{j}+1)}$ has uniformly bounded distortion on $f^{n_{j}+1}(W_{j}^{0})$ not dependent upon $j$. In particular,
$$|f^{n_{j}+1}(L_{j})|/|f^{n_{j}+1}(W_{j})|~~~{\rm and}~~~ |f^{n_{j}+1}(R_{j})|/|f^{n_{j}+1}(W_{j})|$$ are bounded away from zero.
Since $f^{n_{j}}(W_{j}^{0})\subset B_{\gamma}(\cc_{f})$ and  $Sf^{n_{j}+1}(z)<0$, for every $z \in W_{j}^{0}$, the ratio distortion of $f^{n_{j}+1}|_{W_{j}}$ is uniformly bounded.
Thus, taking $V_{j}=W_{j}$, the ratio distortion of $f^{m_{j}}|_{V_{j}}$ is uniformly bounded and $|f^{m_{j}}(V_{j})|=2\delta$ for every $j$.

From now on, we will assume not only that $m_{j}-n_{j}\to\infty$, but also that $f^{i}(W_{j}^{0})\cap J_{0}\ne\emptyset$ for some $n_{j}<i<m_{j}$.

Let $k_{j}$ be the smaller integer $\ell>n_{j}$ such that $f^{\ell}(W_{j}^{0})\cap J_{0}\ne\emptyset$, i.e. $$k_{j}=\min\{\ell>n_{j}\,;\,f^{\ell}(W_{j}^{0})\cap J_{0}\ne\emptyset\}.$$

\begin{Claim}
There is $K>0$ such that $m_{j}-k_{j}\le K$, for every $j \in \mathbb{N}$.
\end{Claim}
\dem[Proof of the claim]
Since $f^{\ell}(x_{j})\notin J_{1}$, for all $n_{j}<j<m_{j}$, there is a connected component $T_{j}$ of $J_{1}\setminus\overline{J_{0}}$ such that $T_{j}\subset f^{k_{j}}(W_{j}^{0})$.
Since $J_{0}$ and $J_{1}$ are nice sets with $J_{0}\subset J_{1}$, it follows that
$$f^{i}(\partial T_{j})\cap J_{0}=\emptyset$$
 for all $i\ge0$. So, if $\ell_{j}\ge0$ is the smaller integer such that
 $$f^{\ell_{j}}(T_{j})\cap J_{0}\ne\emptyset ,$$ i.e. $f^{\ell_{j}}(T_{j})\cap J_{0}(c_{j})\ne\emptyset$ for some $c_{j}\in\cc_{f}$.
  Thus,  $f^{\ell_{j}}(T_{j})\supset J_{0}(c_{j})$. Since $f^{m_{j}}|_{W_{j}^{0}}$ in a diffeomorphism, we get $\ell_{j}\ge m_{j}-k_{j}$.
Thus, from Theorem~\ref{zoom1}, it follows  that  $$4\delta=|f^{m_{j}}(W_{j}^{0})|\ge|f^{m_{j}-n_{j}}(T_{j})|\ge C\lambda^{m_{j}-k_{j}}|T_{j}|\ge C\lambda^{m_{j}-n_{j}}(\gamma_{2}-\gamma_{1}),$$
for every $j \in \mathbb{N}$. Since $\lambda>1$, we necessarily have $m_{j}-k_{j}$ bounded.
\cqd

Using Theorem~\ref{zoom1}, we  conclude that $f^{k_{j}-(n_{j}+1)}$ has uniformly bounded distortion on $f^{n_{j}+1}(W_{j}^{0})$ (not dependent upon $j$).
Since $0\le m_{j}-k_{j}\le K$ and $f^{m_{j}}|_{\overline{W_{j}^{0}}}$ is a diffeomorphism, we obtain that $f^{m_{j}-(n_{j}+1)}$ has uniformly bounded distortion on $f^{n_{j}+1}(W_{j}^{0})$ (also not dependent upon $j$).
Thus,
$$|f^{n_{j}+1}(L_{j})|/|f^{n_{j}+1}(W_{j})|~~~ {\rm and} ~~~|f^{n_{j}+1}(R_{j})|/|f^{n_{j}+1}(W_{j})|$$ are bounded away from zero.
Since $Sf^{n_{j}+1}(z)<0$ the ratio distortion of $f^{n_{j}+1}|_{W_{j}}$ is uniformly bounded  for all $z \in W_{j}^{0}$.
Again, taking $V_{j}=W_{j}$, the ratio distortion of $f^{m_{j}}|_{V_{j}}$ is uniformly bounded and $|f^{m_{j}}(V_{j})|=2\delta$ for all $j$.

Thus, replacing $j$ by a subsequence, we get intervals $V_{j}\subset W_{j}^{0}$ containing $x_{j}$ with the following properties:
 $\inf_{j}|f^{m_{j}}(V_{j})|>0$, the ratio distortion of $f^{m_{j}}|_{V_{j}}$ is uniformly bounded and
  the ratio distortion of $g^{m_j}|_{h({V}_{j})}$ is also uniformly bounded.

By compactness, taking a subsequence, there is an open interval $V$ and a sequence of intervals $x_{j}\in {V}'_{j}\subset{V}_{j}$, $j\ge1$, such that $f^{s_{j}}({V}'_{j})=V$, for all $j$. Thus, $(x,V)$ is a $1$-zooming pair.
Similarly, if $x\in {E}(f)$ there is an interval $V$ such that $(x,V)$ is a central $1$-zooming pair.
\end{proof}

\begin{Lemma} [Explosion of smoothness at expanding points]
\label{zoom7}
Let $f$ and $g$  be  $C^3$  multimodal maps topologically conjugated by $h$,
with no periodic attractors and no neutral periodic points.
Let the conjugacy $h$ be $C^1$ at a point $x$.
If $x\in {\NE}(f)$, then  there is an open interval $V$ such that $h|V$ is $C^r$.
\end{Lemma}

\begin{proof}
By Lemma \ref{zoom6}, if $x\in {\NE}(f)$ there is an interval $V$ such that $(x,V)$ is a $1$-zooming pair.
 Since   $h$ is $C^{1}$ at $x$, then by Lemma \ref{esp1} we have that $h$ is uaa at $x$. Thus,
it follows from Lemma~\ref{esp3} that $h|V$ is a $C^{1+\alpha}$ diffeomorphism. Hence,  by Lemma \ref{esp4}, $h|W$ is a $C^{r}$ diffeomorphism for some $W \subset V$.
\end{proof}

\section{Smooth conjugacy and renormalization intervals}
\label{RRIIII}

In this section we assume that $f$ and $g$ are $C^r$ multimodal maps, with $r \ge 3$ and no periodic attractors nor neutral periodic points.
Furthermore, we assume that $h$ is a topological conjugacy  between  $f$ and $g$  preserving the order of the critical points. We define $$s=\min_{\{c \in C_f\}} \ord_f(c).$$

 \begin{Definition} [Smooth conjugacy  domain]
 For $s \le t \le r$, the \emph{t-smooth conjugacy interval}  $V$ is an open set    $V$ such that $h|V$ is a $C^t$ diffeomorphism. The set $C_f^t \subseteq C_f$
consists of all critical points $c$ such that there is a t-smooth conjugacy open interval $V$ containing   $c \in V$.
 For every $c \in C_f^s$,  the \emph{s-smooth conjugacy maximal interval $J^s(c)$ of} $c$ is the maximal open interval $J^s(c)$ containing $c$ such that $h$ is $C^{s}$ in $J^s(c)$.
The \emph{s-smooth conjugacy  domain $J^s$} is
 $$J^s= \cup_{c \in C_f^s}J^s(c) .$$
We say that a  critical point  $c \in C_f$ is \emph{$s$-recurrent}, if
there is $n=n(c,s) \ge 1$ such that
$J^s(c) \cap f^{n}J^s(c) \ne \emptyset$.
Let $\CR^s \subset C_f$ be the set of all $s$-recurrent critical points.
Let $J^s_R = \cup_{c \in \CR^s} J^s(c)$.
 \end{Definition}

\begin{Lemma} [Spreading smooth conjugacy intervals]
\label{ren0}
Let $h$ be a topological conjugacy  between  $f$ and $g$ and let $s \le t \le r$. Then
\begin{enumerate}
\item if $V$ is a t-smooth conjugacy interval then  $\inter(f(V))$ is a t-smooth conjugacy interval;
\item if $V$ is a t-smooth conjugacy interval then the connected components of $f^{-1}(V)\setminus (f^{-1}(V) \cap C_f)$
are t-smooth conjugacy intervals; and
\item  if $V$ is an s-smooth conjugacy interval then  $f^{-1}(V)$ is an s-smooth conjugacy interval.
\end{enumerate}
Furthermore,
\begin{enumerate}
\item
if $c \in C_f$  and, for some small open interval $V$ containing $c$ and some $n$, $f^n(V) \subset J^s$
then $c \in C_f^s$; and
\item
if $c \in C_f$  and, for some small open interval $V \subset J^r$  and some $n$,  $c \in \inter(f^n(V))$
then $c \in C_f^r$.
\end{enumerate}
\end{Lemma}

\dem   Since $f$ is a multimodal map, the interior of  $f(V)$ is an open interval and for every $x \in f(V)$
there is an open interval $W$ such that $x \in f(W)$ and $f|W$ is a $C^r$ diffeomorphism.
Hence, $h|f(W)=  g \circ  h \circ (f|W)^{-1}$ is a $C^r$ diffeomorphism.

For every   $x  \in f^{-1}(V)\setminus (f^{-1}(V) \cap C_f)$,
there is an open interval $W$ such that $x \in  W$ and $f|W$ is a $C^r$ diffeomorphism.
Hence, $h|W=  (g|_{h(W)})^{-1} \circ  h \circ f$ is a $C^r$ diffeomorphism.

Let $c \in f^{-1}(V) \cap C_f$ and $c'=h(c)$. Let $b=f(c)$ and $b'=g(h(c))$.
Recall that
$f(c + x) = f(c) + \phi (|x|^\alpha)$ and $g(c' + x) = g(c') + \psi (|x|^\alpha)$.
Hence, there is a small open interval $W$ containing $c$, such that for every $c+ x \in W$
$$(g|_{h(W)})^{-1} (y)= c' + (\psi^{-1}  (y - g(c'))^{1/\alpha},$$ if $x \ge 0$
and
$$(g|_{h(W)})^{-1} (y-g(c'))= c' - (\psi^{-1}  (y - g(c'))^{1/\alpha},$$ if $x \le 0$.
Hence,  if $x \ge 0$
$$h|W=  (g|_{h(W)})^{-1} \circ  h \circ f= c' +   [\psi^{-1}  ( - g(c') +  h  ( f(c) + \phi (|x|^\alpha)))]^{1/\alpha}$$
and if $x \le 0$
$$h|W=  (g|_{h(W)})^{-1} \circ  h \circ f= c' -   [\psi^{-1}   ( - g(c') +  h  ( f(c) + \phi (|x|^\alpha)))]^{1/\alpha}.$$

The map $\psi^{-1}  ( - g(c') +  h  ( f(c) + y))$ is a $C^r$ diffeomorphism. Hence, by Taylor's theorem, there is a constant $c$ and $C^r$ diffeomorphism $\theta$
such that   $\psi^{-1}  ( - g(c') +  h  ( f(c) + y))= y (c+ y \theta (y))$.
Therefore,
$$h|W=     x (c+ |x|^\alpha \theta (|x|^\alpha))^{1/\alpha}.$$
Hence, $h$ is a $C^s$ diffeomorphism.
\cqd

Denoting by $p$ and $q$ the points that form the boundary of an interval $V$,  the set $V$ is \emph{dynamically symmetric} if
  either $f(p)=f(q)$ or
$f(p)$ and $f(q)$ form the boundary of $f(V)$.

\begin{Lemma} [Nice $J^s$]
\label{ren1}
If $h$ is a $C^s$ diffeomorphism in an open set $V$ then
the $s$-conjugacy maximal domain $J^s$ is a non-empty and there is $n \ge 0$ such that $ \inter(f^n(V)) \subset J^s$.
Furthermore,
\begin{enumerate}
\item
if $c \in C_f^s$ the set $J^s(c)$ is dynamically symmetric;
\item
for all $c_1,c_2 \in C_f^s$ the  sets $J^s(c_1)$ and $J^s(c_2)$ are either disjoint or equal; and
\item
 the $s$-conjugacy maximal domain $J^s$ is a  nice set.
\end{enumerate}
\end{Lemma}

\dem  Let us assume   that $h$ is a $C^r$ diffeomorphism in an open set $V$.
It follows from Lemma \ref{zoom3} that there is an $n\in\NN$ and $c\in C_{f}$ such that $f^{n}|_{V}$ is a diffeomorphism $C^{r}$ and $c\in \inter(f^n(V))$.
Hence, by Lemma \ref{ren0} (i), $h$ is a $C^r$ diffeomorphism   in $f^n(V)$.
 Hence, $J^s(c)\supset f^{n}(V)$ is   a non-empty closed interval and $c \in C_f^r$.

Let us denote $J^s$ by $J$. Let us denote by $p$ and $q$ the boundary points of $J(c)$. Let us
prove that the interval $J(c)$ is dynamically symmetric, i.e. either
$f(p)=f(q)$, or $f(p)$ and $f(q)$ form the boundary of $f(J(c))$. Let
us suppose, by contradiction, that  there is $z \in \inter J(c)$ that is not a critical point such that
$f(z)=f(q)$ (or, similarly, $f(z)=f(q)$). Let $V_z$ and $V_q$ be
small neighborhoods of $z$ and $q$, respectively, such that $f|_{V_z}$
is a $C^r$ diffeomorphism  and  $f(V_q) \subset f(V_z)$.
Hence, by Lemma \ref{ren0} (i), $h$ is a $C^s$ diffeomorphism   in  $f(V_q) \subset f(V_z)$ and, again by Lemma \ref{ren0},
 $h$ is a $C^r$ diffeomorphism   in $V_q$. Hence, $h$ has a $C^r$ diffeomorphic extension to a
neighborhood of $q$ which is absurd.

By construction, if  $J(c_1) \cap J(c_2) \ne \emptyset$, for some $c_1,c_2 \in C_f$,
then $J(c_1) = J(c_2)$

 Let us prove
that the set $J$  is nice. Let us suppose, by
contradiction, that there is a point $p \in \partial J(c)$ and $n \ge 0$ such that $f^n(p) \in J$ and $f^m(p) \notin J$, for all $0 < m < n$.
Hence, there is a small neighborhood $V$ of $p$ such that
  $f^n(V) \subset J$. By Lemma \ref{ren0} (i) and (iii),  $h$ is a $C^s$ diffeomorphism   in $V$ which is absurd. The proof of case (ii) is
similar.
\cqd

 Given a nice set $J$, let $I(J)$ be the set of all points $x \in I$ whose forward orbit intersects $J$. Let
$D(J)$ be the set of all connected components $G$ of $I(J)$, i.e.
$$I(J)=\bigcup_{G \in D(J)}G.$$
The open intervals $G \in D(J)$ are called the \emph{gaps} of $I(J)$.
We note that the boundary $\partial I(J)$ of $I(J)$  is totally disconnected.

\begin{Lemma}[The basin of attraction of $J^s$]
\label{ren5}
Let $\emptyset \ne J^s \subset  \inter(I)$
For every $G \in D(J^s)$ with $G \cap J^s = \emptyset$, there is $n=n(G) \ge 1$ such that
\begin{enumerate}
\item  $f^n|G$ is a diffeomorphism;
\item
there is $c \in C_f^s $ such that $f^n(G)=J^s(c)$;
\item
 $f^j(G) \cap   J^s = \emptyset$, for every $0 \le j< n$.
\end{enumerate}
\end{Lemma}

\dem
For every $x\in I(J)\setminus J$, let $n(x) > 1$ be  such that $f^n(x)\in J$ and
$f^j(x) \notin J$ for every $0\le j<n$. Let $E=\{x,\ldots,f^{n-1}(x)\}$.
By Lemma \ref{ren1}, $E \cap C_f = \emptyset$ and so there is a small open set $V$ such that $f^n|V$ is a $C^r$ diffeomorphism
and $f^n(V) \subset J$.
Let us prove by contradiction that there is a small open interval $W \subset V$ containing $x$ such that $n(y)=n(x)$ for every $y \in W$.
If there is not a small open interval $W \subset V$ containing $x$ such that $n(y)=n(x)$, for every $y \in W$,
then there is a sequence of points $x_n \in V$ converging to $x$ with $n(x_n)=j <n(x)$.
 Hence, $f^j(x) \in \partial J$. Since $J$ is nice $f^{n-j}(f^j(x))  \cap J = \emptyset$ which is a contradiction.
Let $V=(x,a)$ be the maximal open interval containing $x$ such that $n(y)=n(x)$ for every $y \in V$.
Let us prove, by contradiction, that $f^n(a) \in  \partial J$.
By the above argument, If  $f^n(a) \in  J$ then there is an open interval $W_a$ such that $n(y) = n(a)$ for every $y \in W_a$
which is absurd by maximality of $V$. Hence, for every $x \in I(J)\setminus J$, there is a maximal open interval $G$
such that $n(y) = n(x)$, for every $y \in G$, and $f^n(G) \subset \partial J$. Hence, $f^n|G$ is a  $C^r$ diffeomorphism and
$f^n(G) = J(c)$ for some $c$.
\cqd

\begin{Lemma} [$J^s_R$ is a renormalization domain]
\label{ren6}
Let $\emptyset \ne J^s \subset  \inter(I)$.
For every $c \in C_f^s$,
there is $n(c)$  and  $c'(c) \in C_f^s$ with the following properties:
\begin{enumerate}
\item
 $f^{n(c)}(J^s(c)) \subset \overline{ (J^s(c'(c)))}$;
 \item
 $\partial f^{n(c)}(J^s(c)) \subset \partial J^s(c'(c))$;
 \item
$f^{i}(J^s(c))\cap J^s=\emptyset$, for every $1\le i <n(c)$;
 \item
$J^s_R$  is a renormalization domain;
  \item
  $\cb(J^s_R) \subseteq I(J^s)$ and $\overline{\cb(J^s_R)} = \overline{I(J^s)}$;
\end{enumerate}
\end{Lemma}

 \dem
 By Lemma \ref{ren5}, for every gap $G \in D(J)$ there are
 $n=n(G) \ge 1$ and  $c(G) \in C_f\cap J$ such that $f^n(G)=J(c(G))$,  $f^n|G$ is a
 $C^r$ diffeomorphism and $f^n(G) \cap J = \emptyset$ for every $o \le i < n$.

 For every $c \in C_f$, either (A) $\inter f(J(c)) \cap   \partial I(J) =
 \emptyset$; or (B) $\inter f(J) \cap   \partial I(J) \ne \emptyset$

 Case (A). Since  $\inter f(J(c)) \cap   \partial I(J) =
 \emptyset$, there is  an open interval $K$, that is either a) an interval
 $J(c')$ or b)  a gap $G$,  such that $f(J(c)) \subset \overline{K}$.
 In case a), this lemma follows from noting that $J$ is nice, and so $\partial  f(J(c)) \subset \partial J(c')$. In case b),
 there is $n=n(G) \ge 1$ such that $f^{n+1} J(c) \subset J(c(G))$ and    $f^{i+1} J(c) \cap J = \emptyset$ for every $0 \le i < n$.
Furthermore, since $J$ is nice, $f^{n+1}( \partial J(c)) \subset \partial J(c(G))$  that proves this lemma in case b).

 Case (B).
 Let us suppose that there is a point $x \in
  \partial I(J) \cap \inter f(J(c))$.
 Let $V$ be a small neighborhood contained in $J(c)$
 such that $f|V$ is a $C^r$ diffeomorphism and $x$ is contained in the interior of
 $f(V)$. Since  $\partial I(J)$ is a totally disconnected set, there
 are  gaps $G_y$ and $G_y'$ with a boundary point $y \in f(V)$. Let $z \in V$ be
 such that $f(z)=y$ and
 take a smaller neighborhood $V_0 \subset V$ of $z$ such
 that $f(V_0\setminus  \{z\}) \subset G_y \cup G_y'$.
By Lemma \ref{ren0},  if there is $$w \in f^{n(G_y)+1} (V_0\setminus  \{z\}) \cap \partial J(c(G_y)),$$ then
there is an open interval $W \subset f^{n(G_y)+1} (V_0\setminus  \{z\})$ containing $w$ such that $h|W$ is a $C^s$ diffeomorphism.
Since $w \in \partial J(c(G_y))$, we obtain a contradiction.
Hence, for some $0 \le i < n(G_y)$, there is a critical point $c_y\in C_f$ such that $c_y=f^i(y)$.
Therefore, $J(c(G_y))=J(c(G_y'))$ and $n(G_y)=n(G_y')$.
Since the set of critical points is finite, $\partial I(J) \cap \inter f(J(c_0))$ is also finite
and for every $w \in \partial I(J) \cap \inter f(J(c_0))$, there are gaps  $G_w$ and $G_w'$ with $w \in \partial  G_w \cap  \partial G_w'$
such that
$$J(c(G_w))=J(c(G_w'))= J(c(G_y)) ~~~{\rm  and} ~~~n(G_w)=n(G_w')=n(G_y).$$
Furthermore, since $J$ is nice, $f^{n(G_y)}( \partial J(c)) \subset \partial J(c(G_y))$, that proves this lemma in case (B).

Hence, Lemma \ref{ren6} (i) and (ii) hold. Therefore,  $J^s_R$  is a renormalization domain.
Lemma \ref{ren6} (i) and (ii) also imply  for  every gap $G \subset I(J^s)$ there is a gap $G' \subset \cb(J^s_R)$ such that
  $G\setminus G'$ is either (i) empty or (ii) it is a finite set of points $S_G=G\setminus G'$ with the following properties:
  for every $x \in S_G$ there is $i=i(x)$ and $j=j(x)$ such that (i) $0 \le i < j$, (ii) $f^i(x) \in C_f^s$,
 (iii) $f^i(x) \notin J^s_R$, and (iv) $f^j(x) \in \partial J^s_R$. Hence,  Lemma \ref{ren6} (iv) holds.
 \cqd

\begin{theorem}[Explosion of smoothness]
\label{theoremA}
Let $f$ and $g$ be $C^r$ multimodal maps with $r \ge 3$ and no periodic attractors  and no  neutral periodic points.
Let $h$ be a topological conjugacy  between  $f$ and $g$  preserving the order of the critical points.
If $h$ is $C^{1}$ at a point $p \in \NE(f)$, then either
\begin{enumerate}
\item
$h$ is a $C^{s}$ diffeomorphism  in the full interval $I$ or in its interior $\inter(I)$; or
\item
there is a unique maximal renormalization domain  $J \subseteq I$
such that $h$ is a $C^{s}$ diffeomorphism in $J$.
Furthermore,
\begin{enumerate}
 \item $h$ is a $C^{s}$ diffeomorphism
 in  the basin  of attraction  $\cb(J)$;
 \item  $h$ is not  $C^{s}$ at any open interval contained in $I \setminus  \overline{\cb (J)}$;
\item  $h$ is not $C^{1}$ at any point in $E(f) \cap \partial  \cb(J)$.
 \end{enumerate}
  \end{enumerate}
\end{theorem}

\dem
By Lemma \ref{zoom7}, there is an open interval $W$ such that $h|W$ is $C^s$ and so the  $s$-smooth conjugacy maximal domain $J^s \ne \emptyset$.
If $h$ is not a $C^{s}$ diffeomorphism  in $I$ or $\inter(I)$, then, by Lemma \ref{ren6},  there is
   a renormalization domain $J^s_R$ such that (i) $h|\cb(J^s_R)$ is a  $C^s$ diffeomorphism and (ii)
   there is no open interval $V \subset I \setminus \overline{\cb(J^s_R)} = I \setminus \overline{I(J^s)}$
   such that  $h|$ is a  $C^s$ diffeomorphism.
    Let us prove, by contradiction, that  $h$ is not $C^{1}$ at any point in $E(f) \cap \partial  \cb(J)$.
 By Lemma \ref{zoom7}, if $h$ is  $C^{1}$ at some  point $x \in E(f) \cap \partial  \cb(J)$ then
 there is an open interval $W$ containing $x$ such that $h|W$ is $C^s$ which is a contradiction.
   \cqd

 Theorem \ref{theoremB} below gives  a criterium for non-smoothness of the conjugacy when the conjugacy does not preserve the order of the critical points.
 The \emph{non-critical forward orbit} ${\mathcal O}^+_{nc} (p)$ of $p$ is the set of all points $q$ such that there is $n=n(q) \ge 0$ with the property that $f^n(p)=q$ and
$(f^n)'(p)\ne 0$. The \emph{non-critical omega limit set} $\omega_{nc} (p)$ of $p$ is the set of all accumulation points of  ${\mathcal O}^+_{nc} (p)$.

\begin{theorem} [Implosion of non smoothness]
\label{theoremB}
Let $f$ and $g$ be $C^r$ multimodal maps with $r \ge 3$ and no periodic attractors  and no neutral periodic points.
Let $h$ be a topological conjugacy,  between  $f$ and $g$,  not preserving the order of the critical points $c_f$ and $c_g=h(c_f)$.
The conjugacy $h$ is not $C^{1}$  simultaneously at (i) a  point  belonging to  $E(f) \cap \alpha_{nc} (c_f)$ and a point belonging to  $E(f) \cap \omega_{nc} (c_f)$.
\end{theorem}

If $f$ is a Collet-Eckmann map with negative Schwarzian derivative, then  $E(f) \cap \omega_{nc} (c_f) \ne \emptyset$ and $\alpha_{nc} (c_f)$ contains the Milnor's attractor cycle.

\dem 
Let us prove, by contradiction, that $h$ is not $C^{1}$  at any  point  belonging to  $E(f) \cap \alpha (c_f)$.
If $h$ is   $C^{1}$  at a   point   $x \in E(f) \cap \alpha_{nc} (c_f)$ then,
by Lemma \ref{zoom7}, there is an open interval $V_1$ containing $x$ such that $h|V_1$ is $C^r$.
Since  $x \in  \alpha_{nc} (c_f)$, there is an integer $n$ such that
$c \in int (f^n(V_1))$.
Hence, by Lemma \ref{ren0}, $h$ is a $C^{r}$ diffeomorphism in an open set $V_c$ containing $c$.

If $h$ is   $C^{1}$  at a   point   $x \in E(f) \cap \omega_{nc} (c_f)$ then,
by Lemma \ref{zoom7}, there is an open interval $W_1$ containing $x$ such that $h|W_1$ is $C^r$.
Since  $x \in  \omega_{nc} (c_f)$, there is an open set $W_{f(c)}$ containing $f(c)$ and an integer $n$ such that
$f^n(W_{f(c)}) \subset W_1$ and $f^n|W_{f(c)}$ is a $C^{r}$ diffeomorphism.
Hence, by Lemma \ref{ren0}, $h|W_{f(c)}$ is a $C^{r}$ diffeomorphism.

Since $h$ does not preserve the order of the critical points $c_f$ and $c_g=h(c_f)$, $h$ can not be $C^1$ at $c_f$ and $f(c_f)$ simultaneously which is an absurd.
   \cqd

   \section{$C^r$ smoothness of the conjugacy}

In this section, we prove   Theorem \ref{theoremAAA}.

\label{RRIIII3}

\begin{Lemma} [$K(c') \subseteq J^s_R$ is a renormalization interval]
\label{LemmaA}
Let $h$ be a $C^r$ diffeomorphism in an open set $V_1$.
There is a  maximal renormalization  interval  $K(c') \subseteq J^s_R$
and a puncture set $P(c') \subset K(c')$
such that
\begin{enumerate}
\item $h$ is a $C^r$ diffeomorphism in $K(c') \setminus P(c')$, and
\item $int (V_1 \cap \cb(K(c'))) \ne \emptyset$.
\end{enumerate}
Furthermore, $\partial K(c') \subset E(f)$ and $h$ is not $C^1$ at the boundary $\partial K(c')$ points.
\end{Lemma}

\dem
Using Lemma \ref{zoom3}, there is   a sequence of open sets $V_1,V_2,V_3,\ldots$
 such that (i)  $V_{i+1} \cap C_f \ne \emptyset$; (ii) $f^{n_i} (V_i) \supset  V_{i+1}$,
and  (iii) $|V_i| \to 0$.
Since $C_f$ is finite, (i) there is $c' \in  C_f \cap J$    and (ii) a subsequence $V_{n_1},V_{n_2},V_{n_3},\ldots$
such that $f^{m_i} (V_{n_i}) \supset  V_{n_{i+1}}$, where $m_i=\sum_{j=n_i}^{n_{i+1}-1} n_j$, and (iii) $c' \in V_{n_i}$ for every $i\ge 1$.
By Lemma \ref{ren0}, $h|\inter(f^{m_i} (V_{n_i}))$ is a $C^{r}$ diffeomorphism and
so $h|V_{n_{i+1}}$ is also a $C^{r}$ diffeomorphism.
By Lemma \ref{ren6}, there is a non-empty maximal renormalization interval $J=J^s(c') \subseteq J^s_R$
containing $V_{n_{i}}$ for all $i$.
Let $l$ be the smallest integer such that $F=f^l|J$ is a renormalization of $f$ restricted to $J$.

Let $C$ (possibly empty) be the set of all critical point $c \in C_F$ of $F|J$ such that there is no open interval $V_c \subset J$ with the  property
that $c \in V_c$ and  $h|V_c$ is  a $C^r$ diffeomorphism.
For every $c \in C$, let $\alpha_{nc} (c)$ be the non-critical alpha limit set of $c$ with respect to $F|J$.
Set $\alpha_{nc}(C)= \cup_{c \in C}  \alpha_{nc} (c)$.

 Let us prove that the open connected component $H$  of $J \setminus \alpha_{nc}(C)$ containing $c'$ is a renormalization interval for $F$.
 Let us start proving, by contradiction, that  $H$   is non-empty.
If $H= \emptyset$, there are
(i) $c_1 \in C$, (ii) an open interval $U \in V_{n_1}$ and (iii) an integer $l$ such that $c_1 \in \inter F^{l}(U)$.
By Lemma \ref{ren0},   $c_1 \in C_F^r$ which is absurd.
Take $i_0$ large enough such that, for every $i \ge i_0$,  $c' \in V_{n_i} \subset H$ and $c' \in V_{n_{i+1}} \subset H$. Since $f^{m_i}(V_{n_i}) \supset V_{n_{i+1}}$,
 there is $l_i$ such that  (i) $F^{l_i}(V_{n_i})=f^{m_i}(V_{n_i})$ and (ii)  $F^{l_i}(V_{n_i}) \cap H \ne \emptyset$.
  Since $\alpha_{nc}(C)$ is forward invariant,
$\partial F^{l_i}(H) \subset \alpha_{nc}(C)$. Let us prove, by contradiction, that (i) $\partial F^{l_i}(H) \subset \partial H$
and (ii) $F^{l_i}(H)  \subset \overline{H}$.
If  $F^{l_i}(H)  \not \subset \overline{H}$ then there is $x \in \partial H$ such that $x \in \inter(F^{l_i}(H))$.
Hence, by Lemma \ref{ren0}, $h$ is $C^r$ in an open set containing $x$ which is a contradiction.
Hence,  $F^{l_i}(H)  \subset \overline{H}$ and, by forward invariance of $\alpha_{nc}(C)$, $\partial F^{l_i}(H) \subset \partial H$.
 Thus, $H$ is a renormalization interval for $F$.
Take $k$  the smallest integer such that $F_1=F^{k}|H$ is a renormalization of $F$ restricted to $H$.

 For every open interval $H_1 \subset H$,
let $C_{H_1}$ be the set of all critical point $c \in {H_1}$ of $F_1|H$ such that there is no  open interval $V_c \subset H$ with the  property that $c \in V_c$ and  $h|V_c$ is  a $C^r$ diffeomorphism.
For every $c \in C_{H_1}$, let ${\mathcal O}^-_{nc} (c)$ be the non-critical backward orbit of $c$ with respect to $F_1|H$.
Set ${\mathcal O}^-_{nc} (C_{H_1}) = \cup_{c \in C_{H_1}} {\mathcal O}^-_{nc} (c)$.
 Since the accumulation set of ${\mathcal O}^-_{nc} (C_{H})$ is contained in $\alpha_{nc}(C)$, the set ${\mathcal O}^-_{nc} (C_{H_1})$ is a discrete set of $H$, for every open interval $H_1 \subset H$.

Now, let $H_1 \subset H$ be the maximal open set such that $h|H_1\setminus {\mathcal O}^-_{nc} (C_{H_1})$ is $C^r$.
Either (i) $H_1=H$, or (ii) $H_1 \ne H$ is non-empty.

\emph{Case (i).}  The interval $K(c')=H$ is the maximal interval of renormalization containing $c'$ and $P(c')={\mathcal O}^-_{nc} (C_H)$ is the punctured set of $K(c')$ with the property    that $h|K(c') \setminus P(c')$ is $C^r$. Furthermore, $int (V_1 \cap \cb(K(c'))) \ne \emptyset$.

\emph{Case (ii).} There is $i$ large enough such that $V_{n_i} \subset H_1$ and $F_1^l(V_{n_i}) \cap H_1 \ne \emptyset$.

Let us prove by contradiction that $\partial H_1 \cap {\mathcal O}^-_{nc} (C_H) = \emptyset$.
If $x \in \partial H_1 \cap {\mathcal O}^-_{nc} (C_H)$ take  the smallest $m$  such that $F_1^m(x) \in C_H$.
Let $a$ and $b$ be close enough to $x$ such that (i) either $(a,x)$ or $(x,b)$ is contained in $H_1$, (ii)  $F_1^{m+1}(a)=F_1^{m+1}(b)$, (iii)
$F_1^m|(a,b)$ ,  $F_1^{m+1}|(a,x)$ and $F_1^{m+1}|(x,b)$ are diffeomorphisms. Hence, $(a,b) \subset H_1$ that is a contradiction.

Let us prove  by contradiction  that if  $x \in \partial H_1$ then  $x$ is not contained in the pre-orbit of a critical point.
Take  the smallest $m$  such that  $F_1^m(x)=c$ is a critical point.  Since $c \notin {\mathcal O}^-_{nc} (C_H)$,
there is a small open set $W$ containing $c$ such that $h|W$ is a $C^r$ diffeomorphism. Furthermore,  there is a  small enough open set $V$  such that (i) $V$ contains $x$, (ii) $F_1^m|V$ is a diffeomorphism and (iii) $F_1^m(V) \subset W$. Thus, by Lemma \ref{ren0}, $h|V$ is also a $C^r$ diffeomorphism that  is a contradiction.

Let us prove  by contradiction  that $F_1(\partial H_1) \cap H_1 = \emptyset$.
If   $x  \in \partial H_1$ and $F_1(x) \in H_1$ then there  are small enough open sets $V$ and $W$ such that (i) $V$ contains $x$, (ii) $F_1|V$ is a diffeomorphism because $x$ is not a critical point of $F_1$, (iii) $F_1(V)=W$, (iv) $W \subset H_1$ and so (v) $h|W$ is a $C^r$ diffeomorphism. Hence $h|V$ is also a $C^r$ diffeomorphism that is a contradiction.

There is $i$ large enough such that  $V_{n_i} \subset H_1$ and $c' \in F_1^{k} (V_{n_i})$, for some $k$, and so $F_1^{k}(H_1) \cap H_1 \ne \emptyset$.
Hence, to prove that $H_1$ is a renormalization maximal interval it is enough to prove,
  by contradiction,  that $F_1(\partial H_1) \subset \partial H_1$.
  If   $x  \in \partial H_1$ and $F_1(x) \notin \partial H_1$ and so $F_1(x) \notin \overline{H_1}$,  then (i) there is $y \in H_1$ such that $F_1(y)=x$ and (ii) open intervals $V$ and $W$ with the following properties:  (i) $V$ contains $x$, (ii) $W \subset H_1$ contains $y$, (iii) $F_1|W$ is a diffeomorphism, (iv) $F_1^m(W)=V$. Since $h|W$ is a $C^r$ diffeomorphism, by Lemma~\ref{ren0}, we get that $h|V$ is also a $C^r$ diffeomorphism that is a contradiction.
  Therefore, $K(c')=H_1$ is a renormalization interval containing $c'$ and $P(c')={\mathcal O}^-_{nc} (C_{H_1})$ is the punctured set of $K(c')$  such that $h|K(c') \setminus P(c')$ is $C^r$.
    Furthermore, $int (V_1 \cap \cb(K(c'))) \ne \emptyset$.
\cqd

\dem[Proof of Theorem \ref{theoremAAA}]
By Lemma \ref{zoom7}, there is an open interval $V_1$ such that $h|V_1$ is $C^r$.
If $h$ is not a $C^{r}$ diffeomorphism  in $I \setminus P$, then, by Lemma \ref{LemmaA},  there is a
maximal renormalization  interval
   $K(c')$
and a punctured set $P(c') \subset K(c')$ such that $h$ is a $C^{r}$ diffeomorphism in $K(c') \setminus P(c')$.
By Lemma \ref{ren0}, $h$ is a $C^{r}$ diffeomorphism
 in  the punctured basin  of attraction $\cb_{P}(J(c'))$.

Let $C^r_f$ be the union of all critical points $c\in C_f$ such that $K(c) \ne \emptyset$  is a maximal renormalization  interval  and
  $P(c) \subset K(c)$ is  a punctured subset  such that $h$ is a $C^{r}$ diffeomorphism in $K(c) \setminus P(c)$.
Let $J=\cup_{c \in  C^r_f}  K(c)$ be the maximal renormalization domain and $P=\cup_{c \in  C^r_f}  P(c)$ the punctured set of $J$.
By Lemma \ref{ren0}, $h$ is a $C^{r}$ diffeomorphism
 in  the punctured basin  of attraction $\cb_{P}(J)=\cup_{c \in  C^r_f} \cb_{P}(J(c'))$.

  Let us prove, by contradiction,
$h$ is not a $C^{r}$ diffeomorphism at any open interval $V \subset I \setminus  \overline{\cb (J)}$.
If $h$ is  a $C^{r}$  diffeomorphism at $V$ then, by Lemma  \ref{LemmaA}, there is $c \in C^r_f$ such that
 $int (V \cap \cb(K(c))) \ne \emptyset$ which is a contradiction.

    Let us prove, by contradiction, that  $h$ is not $C^{1}$ at any point in $E(f) \cap \partial  \cb(J)$.
 By Lemma \ref{zoom7}, if $h$ is  $C^{1}$ at some  point $x \in E(f) \cap \partial  \cb(J)$ then
 there is an open interval $W$ containing $x$ such that $h|W$ is $C^s$ which is a contradiction.
   \cqd

\appendix

\section{Properties of multimodal maps}

A periodic point $p$ with period $n\in\NN$ is called a {\em periodic attractor} if   there is an open set $V$ with $p \in \partial V$   such that   $\lim_{j\to+\infty} f^{jn}(V)=p$.
A periodic point $p$ with period $n\in\NN$ is called {\em neutral} if     $|Df^{n}(p)|=1$.
A periodic point $p$  with period $n\in\NN$  is  \emph{weak repelling} if $p$ is neutral and
 there is an open set $V$ with $p \in  V$   such that  $f^{n}|V$ is a diffeomorphism  and  $\lim_{j\to+\infty} (f^{n}|V)^{-n}(x)=p$ for all $x \in V$.
A periodic point $p$  with period $n\in\NN$ is a \emph{repellor} if     $|Df^{n}(p)|>1$.
 Let us denote by $\PR(f)$ the set of all  repellor periodic points of $f$.

\begin{T} [Mañé]
\label{zoom1}
Let $f\colon I\to I$ be a  $C^2$ map
without weak repelling periodic points and such that $\# \Fix(f^{n})<\infty$ for all $n\in\NN$. For every $\gamma>0$, there are $C>0$ and $\lambda>1$ with the following properties:
\begin{enumerate}
\item if $J \subset I$ is an   interval  whose $\omega(J)$ does not intersect any  periodic attractor, and
\item  if $n \in \mathbb{N}$ is such that, for every $0 \le j \le n$, $f^j(J) \cap B_{\gamma}(C_{f})=\emptyset$,
\end{enumerate}
then
$$ \lrd_{f^n}(x,y,z) \le C |f^n(z)-f^n(x)| ~~~\mbox{and}~~~  |f^n(J)| \ge C\lambda^{n}|J|  ,$$
for every $x,y,z \in J$ with $x <y <z$.
\end{T}

\dem
It follows from Mañe's Theorem \cite{Man85} and the fact that the logarithm of a $C^{2}$ map is locally Lipschitz outside the critical set.
\cqd

\begin{Lemma} [Forward capture of a critical point]
\label{zoom3} Let $f\colon I\to I$ be a  $C^2$ map and $\# \Fix(f^{n})<\infty$ for every $n \in\NN$.
For each interval $J\subset I$,  whose $\omega(J)$ does not intersect a periodic attractor, there is $n\in\NN$ such that the interior of  $f^n(J)$ contains a critical point.
\end{Lemma}

\dem Let us
suppose, by contradiction, that $f^n\vert \inter(J)$ is a
diffeomorphism onto its image for every $n\in\NN$.
Since $\omega(J)$ does not intersect a periodic attractor and a $C^{2}$ map does not admit a wandering interval (see \cite{BL89,WS}), there is $k >l>0$ such that
$f^k(J)\cap f^l(J)\neq \emptyset$. The closure $D$ of the set
$\bigcup_{n \ge 0} f^{n(k-l)}(f^l(J))$ is a forward invariant interval for
$f^{(k-l)}$.
Thus, $g=f^{2(k-\ell)}|_{D}$ is monotone map of $D$ into itself.
 Thus,  $\omega_{g}(x)\subset \Fix(g)$ for every $x\in D$.
Since $\# \Fix(g)<\infty$, we get that there is an attracting fixed point $p\in D$ for $g$. Hence, ${\mathcal{O}_{f}}^{+}(p)$ is an attracting periodic orbit for $f$ intersecting $\omega_{f}(J)$, contradicting our hypothesis.
\cqd

\begin{Lemma} [Domain shrinking for iterated local diffeomorphisms]
\label{zoom5} Let $f\colon I\to I$ be a  $C^2$
 map and $\# \Fix(f^{n})<\infty$ for every $n \in\NN$. If $J_1,J_2,... \in I$ is a
sequence of open intervals such that
\begin{enumerate}
\item $\bigcup_{n \ge 1}\omega(J_n)$ does not intersect a periodic attractor and
\item $f^{m_{n}}\vert J_n$ are
diffeomorphisms, with $m_{n}$ tending to $\infty$,
\end{enumerate}
then $|J_n|\to 0$ when $n$ tends to infinity.
\end{Lemma}

\dem Let us suppose, by contradiction, that there is $\delta > 0$
such that $|J_n| > \delta$, for every $n \ge 1$. Since $I$ is
compact, there is an interval $L$ and an infinite  subsequence
$J_{m_1}, J_{m_2}, \ldots$ of intervals   such that   $L \subset J_{m_n}$ for every $n \ge 1$.
Hence, $f^\ell|L$ is a
diffeomorphism, for every $\ell \ge 1$, which, by Lemma
\ref{zoom3}, is a contradiction. \cqd

Following M. Martens \cite{Ma}, a union  $J=\bigcup_{i} J_i$ of  pairwise disjoint open intervals $J_1,J_2, \ldots$ is a
\emph{nice set}, if the forward orbit of the boundaries $\bigcup_{i=1}^l
\partial J_i$ of $J$ do not intersect $J$.

\begin{Lemma} [Nice infinitesimal neighborhoods of critical points]
\label{NICE}
Let $f:I\to I$ be a multimodal map without periodic attractors. For every small $\ve>0$, there is a nice set $J=\bigcup_{c\in\cc_{f}}(p_{c},q_{c})$  $\cn$ such that $c\in(p_{c},q_{c})\subset B_{\ve}(c)$ for all $c\in\cc_{f}$.
\end{Lemma}

We note that, if $\{J_{k}\}$ is the set of connected components of a nice set $J$ then $$J'=\bigcup_{J_{k}\cap\cc_{f}\ne\emptyset}J_{k}$$ is also a nice set.
Let $\cn$ be the collection of all nice set $J=\bigcup_{k}(p_{k},q_{k})$ such that $\cc_{f}\subset J$ and $(p_{k},q_{k})\cap\cc_{f}\ne\emptyset$ for all $k$. We note that, if $U,V\in\cn$ then $U\cap V\in\cn$.

\dem
First, let us show that there is a nice set $J$ such that $\cc_{f}\subset J$. Consider the compact positive invariant set
$$\Lambda=\{x\in I\,;\,f^{j}(x)\notin B_{\ve}(\cc_{f})\ , ~~~ \forall\,j\ge 0 \} .$$
For every $c\in\cc_{f}$, there is a connected component $J_{c,\Lambda} \supset B_{\ve}(c)$  of $I\setminus \Lambda$. Let $J=\bigcup_{c\in\cc_{f}}J_{c,\Lambda}$. Since $\partial J=\bigcup_{c\in\cc_{f}}\partial J_{c, \Lambda}\subset\Lambda$, we get $f^{j}(\partial J)\subset \Lambda$  for every $j\ge0$. Hence, $f^{j}(\partial J)\cap J=\emptyset$ for every $j\ge0$, i.e. $J$ is a nice set and contains $\cc_{f}$.  Thus, $\cn$ is not an empty collection.

If $c\in\cc_{f}$, either $V\supset B_{\ve}(c)$, for all $V\in\cn$, or there exists $V(c)=\bigcup_{\tilde{c}\in\cc_{f}}V_{\tilde{c}}(c)\in\cn$ such that $V_{c}(c)\subset B_{\ve}(c)$ and $\tilde{c}\in V_{\tilde{c}}(c)$ for all $\tilde{c}\in\cc_{f}$.

Let $\cc_{f}^{\ve}$ be the set of $c\in\cc_{f}$ such that $V\supset B_{\ve}(c)$ for all $V\in\cn$. For every $c\in\cc_{f}^{\ve}$, let $H(c)= \inter \bigcap_{J\in\cn}J_{c}$, where $J_{c}$ is the connected component of $J$ containing $c$.  Hence,  $H(c)$ is a  nice interval and
\begin{equation}\label{eqkjkh76jhjh}
H(c)\subset W\mbox{ for all }W\in\cn.
\end{equation}

\begin{Claim}If $c_0\in\cc_{f}$ is non-wandering then $c_0\notin\cc_{f}^{\ve}$  for all $\ve>0$.
\end{Claim}

\dem[Proof of the claim]Let $\ve>0$ and
$c_{0}\in\cc_{f}$ be a non-wandering point.
Hence, take the smallest $n\ge1$ such that $f^{n}(H(c_{0})) \cap {H(c_{0})} \ne \emptyset$.
Either (i) $f^{n}(H(c_{0}))\not\subset \overline{H(c_{0})}$ or (ii) $f^{n}(H(c_{0}))\subset \overline{H(c_{0})}$.

Case (i). Take $q \in  H(c_{0})$ such that $f^{n}(q) \in f^{n}(H(c_{0})) \cap \overline{H(c_{0})}$ and there is a small interval $V_q$ containing $q$ such that $f^{n}|V_q$ is a diffeomorphism.
For every $c\in\cc_{f}$, let $U_{c}$ be the connected component of $\inter(I)\setminus\{q,\cdots,f^{n-1}(q)\}$ containing $c$.
We get that $U=\bigcup_{c\in\cc_{f}}U_{c}$ belongs to $\cn$ and $H(c_{0})\not\subset U_{c_0}$, because $q\in H(c_{0})$ but $q\notin U_{c_0}$, contradicting (\ref{eqkjkh76jhjh}).

Case (ii).  Since $f^{n}(H(c_{0}))\subset \overline{H(c_{0})}$,   $g=f^{n}|_{\overline{H(c_{0})}}$ is a multimodal map and $f^{n}(\partial H(c_{0}))\subset\partial H(c_{0})$.
Since there is no periodic attractor for $g$, there is a periodic point $q\in H(c_{0})$ for the map $g$.
For every $c\in\cc_{f}$, let $U_{c}$ be the connected component of $\inter(I)\setminus\{q,\cdots,f^{m-1}(q)\}$ containing $c$, where $m$ is the period of $q$ with respect to $f$.
We get that $U=\bigcup_{c\in\cc_{f}}U_{c}$ belongs to $\cn$ and $H(c_{0})\not\subset U_{c_0}$, because $q\in H(c_{0})$ but $q\notin U_{c_0}$, contradicting (\ref{eqkjkh76jhjh}).
\cqd

Now, we consider the case of the wandering critical points. Let $\ve>0$ and $c_{0}$ be a wandering critical point. From Lemma~\ref{zoom3},  there is   $n\ge1$ and a non-wandering $\tilde{c}\in\cc_{f}$ such that $\tilde{c}\in f^{n}(H(c_{0}))$. By the claim above, $\tilde{c}\notin\cc_{f}^{\ve}$. Thus, there is  $V=\bigcup_{c\in\cc_{f}}V_{c}\in\cn$ such that $\partial V_{\tilde{c}}\cap f^{n}(H(c_{0}))\ne\emptyset$. Let $q\in H(c_{0})$ be such that $f^{n}(q)\in\partial V_{\tilde{c}}$ and there is a small interval $V_q$ containing $q$ such that $f^{n}|V_q$ is a diffeomorphism.
For every $c\in\cc_{f}$ consider $U_{c}$ the connected component of $V_{c}\setminus\{q,\cdots,f^{t}(q)\}$ containing $c$. Thus  $U=\bigcup_{c\in\cc_{f}}U_{c}\in\cn$ and $H(c_{0})\not\subset U_{c_{0}}$, contradicting (\ref{eqkjkh76jhjh}).
\cqd

\begin{Lemma} [Fatness of  repellors]
\label{Expmap1}
Let $f$  be $C^r$ a multimodal map with $r \ge 3$ and no periodic attractors  and no neutral points.
\begin{enumerate}
\item
If $f$ is  infinitely renormalizable around a critical point $c$,
then there is a renormalization interval  $J(c)$ such that  ${\mathcal O}^-_{nc} (\PR(f))$ is  dense in $\cb (J(c))$.
\item
If $f$ is no renormalizable inside a renormalizable interval $J$,  then $\alpha_{nc}(\PR(f))$ contains $\overline{\cb (J)}$.
\end{enumerate}
\end{Lemma}

\dem
 Let us prove (1). Since $f$ is  infinitely renormalizable around $c$, there is an infinite sequence of intervals $J_1, J_2, \ldots$
such that $J_{n+1}$ is strictly contained in $J_n$ and there is a sequence $m_1,m_2,\ldots$ such that $f^{m_n}|J_n$ is a multimodal map
and $c \in f^{m_n}(J_n)$.
By taking $J_1$ sufficiently small, we assume that for every critical point $c' \in J_1$ with $c' \ne c$,
there is  a sequence $l_1,l_2,\ldots$ such that $m_n l_n < m_{n+1}$  and $c' \in f^{m_n l_n}(J_{n+1})$.
Let $p_n$ be  a periodic point contained in the boundary $\partial J_n$ of $J_n$.
Hence  $p_n$ is a repellor and
the set $S=\cup_{n \ge 1} \alpha_{nc} (p_n)$ contains  $c\in \partial S$.
Let us prove that $S$ is dense in the smallest interval set that contains $S$.
By contradiction, suppose that $S$ is not a dense set. Hence, there is an open interval $K$ such that $K	 \subset J_1 \setminus S$ and
$\partial K \subset S$. By forward invariance  of $S$ under $f^{m_1}$, $f^{m_1k}(K) \subset J_1 \setminus S$ and
$\partial f^{m_1k}(K) \subset S$ for every $k$.
By Lemma \ref {zoom3}, there is $k_1$ such that $f^{m_1k_1}(K)$ contains some critical point $c' \in J_1$.
Hence, there is $n$ large enough and $l_n$ such that $f^{m_n l_n}(J_{n+1}) \subset f^{m_1k_1}(K)$.
Hence, there is $k_2$ such that     $f^{m_{n+1}}(J_{n+1}) \subset f^{m_1k_2}(K)$.
Since $c \in  f^{m_{n+1}}(J_{n+1})$, we get
 $c \in f^{m_1k_2}(K)$. Noting that  $p_n$ converges to $c$, we obtain that $f^{m_1k_2}(K)$ contain some $p_n$, for $n$ large,
which contradicts that $f^{m_1k_2}(K) \subset J_1 \setminus S$.
Hence, $S$ is dense in the smallest interval set that contains $S$. Since $c\in \partial S$ is a turning point,
 $S$ is dense in a small neighborhood of $c$.
 Hence, there is a renormalization interval $J(c)$, small enough, containing $c$ that is contained in the closure of $S$.

Let us prove (2). Since $J$  is a renormalization interval, there is $m$ such that $f^m|J$ is a  multimodal map.
  Let  $p \in J$ be a periodic repellor with period $k$ of the map $f^m|J$.
  Since $\alpha_{nc} (p)$ is a closed set, it is enough to prove that  $\alpha_{nc} (p)$ is dense in $J$.
By contradiction, suppose that $\alpha_{nc} (p)$ is not a dense set.
 Hence, there is an open interval $K$ such that $K	 \subset J \setminus \alpha_{nc} (p)$ and
$\partial K \subset \alpha_{nc} (p)$.
By forward invariance  of $\alpha_{nc} (p)$ under $f^{m}$, $f^{mk}(K) \subset J_1 \setminus \alpha_{nc} (p)$ and
$\partial f^{mk}(K) \subset \alpha_{nc} (p)$ for every $k$.
 By Lemma \ref{zoom3}, there is a sequence $k_1,k_2,\ldots$ such that $K_n=f^{mk_n}(K)$ contains some critical point $c_n \in J$.
 Since, the set of critical points in $J$ is finite, there is a critical point $c \in J$ and $k_{l_1} < k_{l_2}$
 such that $K_{l_1}$ and  $K_{l_2}$ contain the critical point $c \in J$.
 Hence, $K_{l_1} \cap K_{l_2} \ne \emptyset$.
 Since
 $$\partial K_{l_1} \subset \alpha_{nc} (p) ~~~,~~~\partial K_{l_1} \subset \alpha_{nc} (p)~~~,~~~ K_{l_1} \cap \alpha_{nc} (p) = \emptyset
 ~~~{\rm and}~~~
 K_{l_1} \cap \alpha_{nc} (p) = \emptyset ,$$
 we obtain that $K_{l_1} = K_{l_2}$. In particular, $f^{m (k_l-k_{l_1})}|K_{l_1}$ is a multimodal and $K_{l_1}$ is strictly contained in $J$
 which  contradicts that $f$ is no renormalizable inside of the renormalizable interval $J$.
 Hence,  $\alpha_{nc} (p)$  contains the closure of $J$. Hence, by definition of alpha limit, $\alpha_{nc} (p)$ contains
 $\overline{\cb (J)}$.
\cqd

%

\end{document}